\documentclass{article}[12pt]
\usepackage{amssymb}
\usepackage{eurosym}
\usepackage{epsfig}
\usepackage{amsmath}
\usepackage{amsfonts}
\usepackage{pgf,tikz}
\usepackage{enumerate}
\usepackage{cite}
\usepackage{xcolor, colortbl}
\usepackage[top=2cm, bottom=2.5cm, left=2.5cm, right=2.5cm, headsep=1cm]{geometry}

\newcommand{\azul}[1]{{\color{blue}{#1}}}

\usepackage[colorlinks=true, linkcolor=black, citecolor=black, urlcolor=black]{hyperref}

\newtheorem{theorem}{\sc Theorem}[section]

\newtheorem{definition}{\sc Definition}[section]
\newtheorem{remark}{\sc Remark}[section]
\newtheorem{corollary}{\sc Corollary}[section]
\newtheorem{example}{\sc Example}[section]

\def\qed{\hbox to 0pt{}\hfill$\rlap{$\sqcap$}\sqcup$\medbreak}

\title{Multiple non-collision periodic solutions for second order singular differential systems with parameters}

\author{Laura Mª Fern\'andez--Pardo and Jorge Rodr\'iguez--L\'opez} 

\date{}
\begin{document}
 \maketitle

%

\medbreak

\noindent {\it Abstract.} In this paper, new existence and multiplicity results for positive periodic solutions of singular second-order problems under sublinear and superlinear growth assumptions at infinity are established. Two aspects of our results should be highlighted in comparison with the existing literature. On the one hand, they allow for the consideration of systems exhibiting mixed behavior: sublinear in one component and superlinear in another. On the other hand, in the superlinear case, the existence of up to four positive solutions is established, whereas only two were obtained in previous works. Our approach is based on a recently established vectorial version of the Krasnosel'ski\u{\i}--Guo fixed point theorem in product spaces.

\medbreak

\noindent     \textit{2020 MSC:} 34C25, 34B16, 47H10.

\medbreak

\noindent     \textit{Key words and phrases.} Positive solutions; Nonlinear singular systems; Periodic solutions; Superlinear nonlinearity; Sublinear nonlinearity.

\section{Introduction}       

In this paper we establish the existence and multiplicity of positive solutions for a system with parameters of the following form
\begin{equation}
	\label{eq1.1}
	\left\{
	\begin{aligned}
		& x_1'' + a_1(t)x_1 = (-1)^{n_1}\lambda_1 \bigl( h_1(t) f_1(x_1,x_2) + g_1(t,x_1,x_2) \bigr), \quad t \in [0,T], \\
		& x_2'' + a_2(t)x_2 = (-1)^{n_2}\lambda_2 \bigl( h_2(t) f_2(x_1,x_2) + g_2(t,x_1,x_2) \bigr), \quad t \in [0,T], \\
		& x_1(0) - x_1(T) = 0 = x_1'(0) - x_1'(T), \\
		& x_2(0) - x_2(T) = 0 = x_2'(0) - x_2'(T),
	\end{aligned}
	\right.
\end{equation}
where $T>0, \,\lambda_i>0$, $n_i\in\{0,1\}$, $a_i\in L^{p_i}(0,T)$ with $1\leq p_i\leq +\infty$, $i=1,2$, $h_1,h_2\in L^1(0,T)$, and $f_1,f_2: \mathbb{R}^+\times \mathbb{R}^+\setminus \{(0,0)\}\to\mathbb{R}^+$, $g_1,g_2:[0,T]\times\mathbb{R}^+\times\mathbb{R}^+\to\mathbb{R}^+$ are continuous functions ($\mathbb{R}^+:=[0,\infty)$). Note that the nonlinearities $f_i(x_1,x_2)$, $i=1,2$ are not necessarily defined at $(x_1,x_2)=(0,0)$. In fact we are mainly interested in the existence of a singularity at this point, i.e.,
\[\lim_{(x_1,x_2)\to (0,0)} f_i(x_1,x_2)=+\infty.\]
A canonical example is given by the case where the nonlinearities $f_i$, $i=1,2$ are of the form
\[f_i(x_1,x_2)=\frac{1}{\sqrt{\bigl(x_1^2+x_2^2\bigr)^\alpha}},\]
with $\alpha>0$. This type of singular terms has been studied, for instance, in \cite{CTZ, CZ, ft, Lin, WCS, W}.

By a \emph{positive} solution of \eqref{eq1.1}, we mean a $T$-periodic function $(x_1,x_2)$ satisfying the differential equations almost everywhere and such that $x_1(t),x_2(t)>0$ for all $t\in[0,T]$. This is, in particular, a \emph{non--collision} solution, since no component vanishes. The question of existence of collisionless periodic orbits for systems with singularities has attracted much attention of many mathematicians over many years, such as \cite{CTZ, CZ, ft,  Lin, Torres2, WCS, W}. Even in the scalar case, the existence of periodic solutions for singular problems has
been a classical line of research for many authors, see \cite{Pino, Pino2, Haberts, Rachun, Torres, jiang,jiang2} and the references therein. This is likely due to the many applications of such problems, including gravitational and electromagnetic forces, as illustrated in the monograph by Torres \cite{To}.

Although our formulation of problem \eqref{eq1.1} covers the cases of both \emph{attractive} and \emph{repulsive} singularities (or a component-wise combination of them), depending on the values of $n_i\in\{0,1\}$, we will focus on repulsive singularities throughout the manuscript, i.e., the case $n_1=n_2=0$. The remaining possibilities can be similarly considered, as explained in Remark~\ref{rem1} below.

Our main assumptions on the nonlinearities $g_1$ and $g_2$ concern their asymptotic behavior at infinity. More precisely, we impose either sublinear or superlinear growth on each of the functions $g_1$ and $g_2$, independently. This growth is expressed in terms of the limit at infinity of the quotient $g_i(t,x_1,x_2)/x_i$, which is assumed to be zero (sublinear case) or infinity (superlinear case). Roughly speaking, our main results are the following:
\begin{enumerate}
	\item Problem \eqref{eq1.1} has at least one positive solution for any $\lambda_1,\lambda_2>0$ if both $g_1$ and $g_2$ are sublinear at infinity.
	\item Problem \eqref{eq1.1} has at least two positive solutions for any $\lambda_1>0$ and $\lambda_2>0$ sufficiently small if $g_1$ is sublinear and $g_2$ is superlinear at infinity.
	\item Problem \eqref{eq1.1} has at least four positive solutions for $\lambda_1,\lambda_2>0$ sufficiently small provided that both $g_1$ and $g_2$ are superlinear at infinity.
\end{enumerate}

Our approach is grounded in topological fixed point theory. More precisely, we apply a recently established component-wise Krasnosel'ski\u{\i}--Guo type fixed point theorem in product spaces \cite{LFP_JRL}. The idea of extending the classical Krasnosel'ski\u{\i} fixed point theorem to product spaces in this direction was first introduced by Precup \cite{PrecupFPT}. Within its major advantages over the classical Krasnosel’ski\u{\i} fixed point theorem, it highlights the fact that this formulation allows each component of the operator to exhibit distinct compressive or expansive behavior—a hybrid feature that has received comparatively little attention in the literature—while simultaneously ensuring the existence of fixed points in which both components are nontrivial. Moreover, it also enables the derivation of new multiplicity results.

It is worth noting that the existence of positive periodic solutions for singular systems with repulsive singularities and parameters has already been investigated in \cite{W}, where the classical Krasnosel'ski\u{\i}--Guo fixed point theorem is employed. In fact, the aforementioned reference deals not only with repulsive singularities but also with attractive ones, although the type of singularity is assumed to be the same in all components of the system. Their main results guarantee the existence of at least one positive solution when the nonlinearities are sublinear, and at least two solutions when they are superlinear and the parameters are sufficiently small.
	
In contrast, our approach allows for greater flexibility. Each nonlinearity $g_i$ may exhibit either sublinear or superlinear behavior, and the results can be adapted to hybrid situations in which different components of the system involve singularities of a different nature, for instance, one repulsive and another attractive, as previously mentioned. Furthermore, stronger multiplicity results are obtained. More precisely, while \cite{W} guarantees the existence of two solutions in the superlinear case, our approach ensures the existence of up to four distinct periodic solutions when both nonlinearities are superlinear, thereby improving the latter result by two additional solutions. Finally, our results also establish a stronger positivity property: all components of each solution are strictly positive.

%
%

\section{Preliminaries}    

We first recall here our problem of interest, namely
\begin{equation}
	\label{eq2.2}
	\left\{
	\begin{aligned}
		& x_1'' + a_1(t)x_1 =\lambda_1 \bigl( h_1(t) f_1(x_1,x_2) + g_1(t,x_1,x_2) \bigr), \quad t \in [0,T], \\
		& x_2'' + a_2(t)x_2 =\lambda_2 \bigl( h_2(t) f_2(x_1,x_2) + g_2(t,x_1,x_2) \bigr), \quad t \in [0,T], \\
		& x_1(0) - x_1(T) = 0 = x_1'(0) - x_1'(T), \\
		& x_2(0) - x_2(T) = 0 = x_2'(0) - x_2'(T),
	\end{aligned}
	\right.
\end{equation}
where $T>0, \,\lambda_i>0$, $a_i\in L^{p_i}(0,T)$ with $1\leq p_i\leq +\infty$, $i=1,2$, $h_1,h_2:[0,T]\to \mathbb{R}$, $h_1,h_2\in L^1(0,T)$, $f_1,f_2: \mathbb{R}^+\times \mathbb{R}^+\setminus \{(0,0)\}\to\mathbb{R}$ and $g_1,g_2:[0,T]\times\mathbb{R}^+\times\mathbb{R}^+\to\mathbb{R}$ are nonnegative continuous functions.

Hereafter, we will assume that the following condition holds:

\begin{enumerate}
	\item[$(A_1)$]  For each $i\in\{1,2\}$, the Hill equation $x'' + a_i(t) x = 0$ admits only the trivial solution, and its Green's function $G_i(t,s)$ associated with the periodic boundary conditions $x(0) - x(T) = 0 = x'(0) - x'(T)$ is positive and continuous for all $(t,s)\in[0,T] \times [0,T]$.
\end{enumerate}

Note that when $a_i$ is non-constant, a classical $L^p$-criterion ensures that the associated Green's function has a constant sign; that is, it is either strictly positive or strictly negative on $[0,T]\times[0,T]$ (see \cite{Torres}). Several refinements and extensions of this result can be found in \cite{CC_London,CCL}.

In the constant case, when $a_i(t)\equiv k^2$ with $0<k<\pi/T$, the Green's function can be computed explicitly and is known to be positive (see \cite{Cabada,W}), and hence condition $(A_1)$ is satisfied.

Under assumption $(A_1)$, system \eqref{eq2.2} can be equivalently reformulated as the following system of Hammerstein-type integral equations:
\begin{equation}
		\label{eq2.3}
\begin{array}{ll}
	x_1(t) &= \lambda_1 \displaystyle\int_{0}^{T} G_1(t,s) \Bigl( h_1(s) f_1(x_1(s),x_2(s)) + g_1(s,x_1(s),x_2(s)) \Bigr) \, ds 
	=: \mathcal{T}_1(x_1,x_2)(t), \quad t\in[0,T], \\[10pt] 
	x_2(t) &= \lambda_2 \displaystyle\int_{0}^{T} G_2(t,s) \Bigl( h_2(s) f_2(x_1(s),x_2(s)) + g_2(s,x_1(s),x_2(s)) \Bigr) \, ds 
	=: \mathcal{T}_2(x_1,x_2)(t), \quad t\in[0,T],
\end{array}
\end{equation}
where $G_i$, $i=1,2$, denote the associated Green's functions. Our aim is to find solutions of \eqref{eq2.2} as fixed points of the operator $\mathcal{T} = (\mathcal{T}_1, \mathcal{T}_2) : X \times X \to X \times X$, where $X$ denotes the Banach space of continuous functions $\mathcal{C}[0,T]$ endowed with the sup-norm 
\[
\|u\|_\infty := \max_{t \in [0,T]} |u(t)|.
\]

To this end, we shall employ a \textit{vectorial version of Krasnosel'ski\u{\i}'s fixed point theorem} in the product of two cones, which imposes compression--expansion conditions on the operator in terms of norms in a component-wise manner \cite{LFP_JRL}. We state this result below; first, however, for the reader’s convenience, we briefly recall the notion of a cone. A nonempty, closed, and convex subset $K$ of a normed linear space $X$ is called a \textit{cone} if it satisfies $\lambda u \in K$ for all $u \in K$ and all $\lambda \ge 0$, and $K \cap (-K) = \{0\}$.

\begin{theorem}\label{main_th}
 Let $K_1$ and $K_2$ be two cones in the normed spaces $(X_1,\|\cdot\|_1)$ and $(X_2, \|\cdot\|_2)$, respectively. We denote $K:=K_1\times K_2$ the corresponding cone in $X=X_1\times X_2$, and for $r,R\in\mathbb{R}^2$ with $0<r_i<R_i$ ($i=1,2$) 
	 \[\overline{K}_{r,R}:=\left\{ x=(x_1,x_2)\in K: r_i\leq\left\|x_i\right\|_i\leq R_i \text{ for } i=1,2\right\}.\]
	Assume that $\mathcal{T}=(\mathcal{T}_1,\mathcal{T}_2):\overline{K}_{r,R}\rightarrow K$ is a compact map and for each $i\in\{1,2\}$ either of the following conditions holds in $\overline{K}_{r,R}$:
	\begin{enumerate}
		\item[(a)] $\left\|\mathcal{T}_i x\right\|_i\geq \left\|x_i\right\|_i$ if $\left\|x_i\right\|_i=r_i$ and  $\left\|\mathcal{T}_i x\right\|_i\leq \left\|x_i\right\|_i$ if $\left\|x_i\right\|_i=R_i$; or
		\item[(b)] $\left\|\mathcal{T}_i x\right\|_i\leq \left\|x_i\right\|_i$ if $\left\|x_i\right\|_i=r_i$ and  $\left\|\mathcal{T}_i x\right\|_i\geq \left\|x_i\right\|_i$ if $\left\|x_i\right\|_i=R_i$. 
	\end{enumerate}
	Then $\mathcal{T}$ has at least one fixed point $x=(x_1,x_2)\in K$ such that $r_i\leq \left\|x_i\right\|_i\leq R_i$ ($i=1,2$).
\end{theorem}

Using standard terminology, a component $\mathcal{T}_i$ of the operator $\mathcal{T}$ is termed \textit{compressive} if condition $(a)$ holds, and \textit{expansive} if condition $(b)$ is satisfied.  

We will apply Theorem~\ref{main_th} to the restriction of operator $\mathcal{T}$ to the cone $K := K_1 \times K_2$ where
\begin{equation}
	\label{eq2.5}
K_i:=\left\{u\in X:\min_{t\in[0,T]}u(t)\geq c_i\left\|u\right\|_{\infty} \right\}, \qquad i=1,2,
\end{equation}
with
\[  m_i:=\min_{t,s\in[0,T]}G_i(t,s),\quad M_i:=\max_{t,s\in[0,T]}G_i(t,s) \quad \text{ and } \quad  c_i:=\frac{m_i}{M_i}.\]

Under assumption $(A_1)$, it is readily verified that the operator $\mathcal{T}$ leaves the cone $K := K_1 \times K_2$ invariant; see similar reasoning in \cite{Torres, W}. Moreover, by the continuity of the Green's functions and the nonlinearities, $\mathcal{T}$ is \textit{completely continuous}, that is, it is continuous and maps bounded sets into relatively compact sets. This property is a standard consequence of the Ascoli--Arzel\`a theorem.

\begin{remark}
	\label{rem1}
	Observe that, if for some component $i \in \{1,2\}$ one wishes to consider an attractive singularity (i.e., $n_i = 1$ in \eqref{eq1.1}), then condition $(A_1)$ must be replaced in that component by the following condition to ensure that the operator $\mathcal{T}$ maps the cone $K := K_1 \times K_2$ into itself:
	\begin{enumerate}
		\item[$(A_1)^*$] \emph{The Hill equation $x'' + a_i(t)x = 0$ admits only the trivial solution, and its Green's function $G_i(t,s)$ associated with the periodic boundary conditions $x(0) - x(T) = 0 = x'(0) - x'(T)$ is negative and continuous for all $(t,s)\in[0,T]\times[0,T]$.}
	\end{enumerate}
As a consequence, we anticipate that all results of this work, centered on the case of repulsive singularities in both components, can be easily adapted to situations where one or both components exhibit an attractive singularity at the origin, by simply replacing $(A_1)$ with $(A_1)^*$.
\end{remark}

\section{Main results}

In this section, we present our main results on the existence and multiplicity of solutions for the singular system \eqref{eq2.2}. More precisely, by combining asymptotically sublinear and superlinear conditions on the nonlinearities $g_i$ ($i=1,2$), we establish the existence of at least one, two, and up to four distinct solutions. These results are obtained under the assumption that the following two instrumental conditions hold:
\begin{enumerate}
	\item[$(H_1)$] $f_i:\mathbb{R}^+\times\mathbb{R}^+\setminus\{(0,0)\}\to\mathbb{R}^{+}$ is strictly decreasing in both arguments, $i=1,2$.
	\item[$(H_2)$] $h_i:[0,T]\to\mathbb{R}^+$ satisfies
	\[
	\left\|h_i\right\|_{L^1}:=\int_0^T h_i(t)\,dt>0,\quad i=1,2.
	\]
\end{enumerate}

We next introduce the notions of sublinearity and superlinearity at infinity upon which the main results are based.
\begin{definition}
	Let $i \in \{1,2\}$ and let $g_i:[0,T]\times\mathbb{R}^+\times\mathbb{R}^+\to \mathbb{R}^{+}$. 
We say that:
\begin{enumerate}
	\item[$(i)$] $g_i$ is \emph{sublinear at infinity} if 
	\begin{equation}
			(g_i)_\infty := \lim_{x_i\to\infty} \frac{g_i(t,x_1,x_2)}{x_i} = 0, \quad \text{uniformly with respect to } t \in [0,T] \text{ and } x_j \ge 0, \, (j \neq i).
	\end{equation}
	\item[$(ii)$] $g_i$ is \emph{superlinear at infinity} if 
		\begin{equation}
		(g_i)_\infty := \lim_{x_i\to\infty} \frac{g_i(t,x_1,x_2)}{x_i} = \infty, \quad \text{uniformly with respect to } t \in [0,T] \text{ and } x_j \ge 0, \, (j \neq i).
	\end{equation}
\end{enumerate}
\end{definition}

\subsection{Sublinear nonlinearities}

We state our first existence theorem as follows.

\begin{theorem}
	\label{th_sol}
	Let $(A_1)$, $(H_1)$, and $(H_2)$ hold.  Assume that $g_i$ is sublinear at infinity for each $i=1,2$. 
	
	Then, for all $\lambda_i>0$, $i=1,2$, system \eqref{eq2.2} has at least one positive periodic solution.
\end{theorem}

\noindent
{\bf Proof.} Consider the cone $K := K_1 \times K_2$, with $K_1$ and $K_2$ defined in \eqref{eq2.5}, and the operator $\mathcal{T} = (\mathcal{T}_1, \mathcal{T}_2)$ introduced in \eqref{eq2.3}.  
We aim to find $(\tilde{r}_1, \tilde{r}_2), (\tilde{R}_1,\tilde{R}_2) \in \mathbb{R}^2$ with $0 < \tilde{r}_i < \tilde{R}_i$ ($i=1,2$) such that Theorem \ref{main_th} applies to $\mathcal{T}:\overline{K}_{(\tilde{r}_1, \tilde{r}_2),(\tilde{R}_1,\tilde{R}_2)} \to K$. This operator is compact and maps $\overline{K}_{(\tilde{r}_1, \tilde{r}_2),(\tilde{R}_1,\tilde{R}_2)}$ into $K$.

Let $\lambda_i>0$, $i=1,2$. Since the functions $f_i$ are decreasing and $g_i$ are sublinear at infinity, there exist $R_i\equiv R_i(\lambda_i) > 0$ large enough satisfying
\begin{equation}
	\label{eq3.8_revised}
	2\lambda_1 M_1 \|h_1\|_{L^1} f_1(c_1 R_1, 0) \leq R_1, \quad 
	2\lambda_2 M_2 \|h_2\|_{L^1} f_2(0, c_2 R_2) \leq R_2,
\end{equation}
and
\begin{equation}
	\label{eq3.9_revised}
	g_i(t,x_1,x_2) \leq \frac{x_i}{2 \lambda_i M_i T}, \quad 
	\text{for } x_i \ge c_i R_i,\ x_j \ge 0,\ t \in [0,T],\ (j \neq i),\ i=1,2.
\end{equation}

Next, select $r_1\equiv r_1(\lambda_1,\lambda_2)\in (0,R_1)$ and $r_2\equiv r_2(\lambda_1,\lambda_2) \in (0,R_2)$ such that
\begin{equation}
	\label{eq3.10_revised}
	\lambda_1 m_1 f_1(r_1,R_2) \left\|h_1\right\|_{L^1} \geq r_1, \quad
	\lambda_2 m_2 f_2(R_1,r_2) \left\|h_2\right\|_{L^1} \geq r_2.
\end{equation}
This selection is possible because each function $f_i$ is nonnegative and strictly decreasing.

We now verify condition $(a)$ of Theorem \ref{main_th} for $i=1$ (the case $i=2$ is entirely analogous) for the restriction of $\mathcal{T}$ to the set $\overline{K}_{(r_1,r_2),(R_1,R_2)}$.

Let $ x = (x_1,x_2) \in\overline{K}_{(r_1,r_2),(R_1,R_2)}$ with $ \|x_1\|_\infty = R_1$.  
Then $c_1 R_1 \le x_1(t) \le R_1$ and $0 \le x_2(t) \le R_2$, for all $t \in [0,T]$.
Using $(H_1)$, \eqref{eq3.8_revised}, and \eqref{eq3.9_revised}, we obtain
\begin{align*}
	\mathcal{T}_1 x(t) 
	&= \lambda_1 \int_0^T G_1(t,s)\bigl( h_1(s)\, f_1(x_1(s),x_2(s)) + g_1(s,x_1(s),x_2(s)) \bigr)\, ds \\
	&\leq \lambda_1 \int_0^T G_1(t,s)\left( h_1(s)\, f_1(x_1(s),x_2(s)) + \frac{x_1(s)}{2\lambda_1 M_1 T} \right)\, ds \\
	&\le \lambda_1 M_1 \|h_1\|_{L^1} f_1(c_1 R_1,0) + \frac{R_1}{2} \leq R_1 = \|x_1\|_\infty,
\end{align*}
for all $ t \in [0,T] $, so that $\|\mathcal{T}_1 x\|_\infty \leq \|x_1\|_\infty$ if $\|x_1\|_\infty=R_1$.

Similarly, let $ x = (x_1,x_2) \in \overline{K}_{(r_1,r_2),(R_1,R_2)}$ with $ \|x_1\|_\infty = r_1 $.  
Then $c_1 r_1 \le x_1(t) \le r_1$ and $c_2 r_2 \le x_2(t) \le  R_2$ for all $ t \in [0,T]$. By $(H_1)$ and \eqref{eq3.10_revised} we have
\[
\mathcal{T}_1 x(t) 
\ge \lambda_1 \int_0^T G_1(t,s)\, h_1(s)\, f_1(x_1(s),x_2(s))\, ds
\ge \lambda_1 m_1 f_1(r_1, R_2) \int_0^T h_1(s)\, ds
\geq r_1 = \|x_1\|_\infty,
\]
for all $t \in [0,T]$. Hence $\|\mathcal{T}_1 x\|_\infty \geq \|x_1\|_\infty$ if $ \|x_1\|_\infty = r_1$.  This completes the verification of condition $(a)$ for $\mathcal{T}_1$.  

Theorem \ref{main_th} guarantees the existence of a fixed point 
$x^1=(x_1^1,x_2^1)$ of $\mathcal{T}$ satisfying 
$r_i \leq \|x_i^1\|_\infty \leq R_i$, $i=1,2$, with both components positive, 
thus providing the desired solution of system \eqref{eq2.5}. \qed

\subsection{Asymmetric growth: sublinear-superlinear nonlinearities}

If, instead of sublinearity in both components, superlinearity is assumed in one of them, Theorem~\ref{main_th} guarantees the existence of two distinct solutions for certain configurations of the parameters $(\lambda_1,\lambda_2)$. This yields our second main result, stated below.

\begin{theorem}
	\label{th_two_sol}
	Let $(A_1)$, $(H_1)$ and $(H_2)$ hold. Assume that $g_1$ is sublinear and $g_2$ is superlinear at infinity. 
	
	Then, for each $\lambda_1>0$, there exist $\overline{\lambda}_2(\lambda_1)>0$ such that system \eqref{eq2.2} has, at least, two distinct positive periodic solutions for all $0<\lambda_2<\overline{\lambda}_2(\lambda_1)$.
\end{theorem}

\noindent
{\bf Proof.} We will find $\tilde{r}_1,\tilde{r}_2, \tilde{r}_2^*, \tilde{R}_1,\tilde{R}_2\in\mathbb{R}$ satisfying  $0<\tilde{r}_1<\tilde{R}_1$ and $0<\tilde{r}_2<\tilde{r}_2^*<\tilde{R}_2$, such that Theorem \ref{main_th} applies to the restriction of $\mathcal{T}$ to $\overline{K}_{(\tilde{r}_1,\tilde{r}_2),(\tilde{R}_1,r^*_2)}$ and $\overline{K}_{(\tilde{r}_1,\tilde{r}_2^*),(\tilde{R}_1,\tilde{R}_2)}$.

Fix $\lambda_1>0$ and $r_2^*>0$. Since $f_1$ is decreasing and $g_1$ is sublinear at infinity, there exists $R_1\equiv R_1(\lambda_1)>0$ such that
\begin{equation}
	\label{eq3.11}
	2\lambda_1 M_1 \|h_1\|_{L^1} f_1(c_1 R_1,0) \leq R_1,
\end{equation}
and
\begin{equation}
	\label{eq3.12}
	g_1(t,x_1,x_2) \leq \frac{1}{2 \lambda_1 M_1 T} x_1, 
	\quad \text{for all } x_1 \ge c_1 R_1,\, x_2 \ge 0, \text{ and } t \in [0,T]\azul{.}
\end{equation}

On the other hand, since $g_2$ is superlinear at infinity, there exists $R_2\equiv R_2(\lambda_2) > r_2^*$ such that
\begin{equation}
	\label{eq3.13}
	g_2(t,x_1,x_2) \geq\frac{1}{c_2 \lambda_2 m_2 T} x_2, 
	\quad \text{for all } x_1 \ge 0,\, x_2 \ge c_2 R_2, \text{ and } t \in [0,T].
\end{equation}
Let now $\lambda_2\in\bigl(0,\overline{\lambda}_2(\lambda_1)\bigr)$, being
\begin{equation}
	\label{eq3.14}
	\overline{\lambda}_2(\lambda_1):=r^*_2\left[M_2 \left(\|h_2\|_{L^1} f_2(0,c_2r^*_2)+T\max_{[0,T]\times[0,R_1(\lambda_1)]\times[c_2r^*_2, r^*_2]} g_2(t,x_1,x_2)\right)\right]^{-1}.
\end{equation}
For this $\lambda_2$ we choose $r_1\equiv r_1(\lambda_1,\lambda_2)\in(0,R_1)$ and $r_2\equiv r_2(\lambda_1, \lambda_2)\in(0,r^*_2)$ such that
\begin{equation}
	\label{eq3.15}
	\lambda_1 m_1 f_1(r_1,R_2)\left\|h_1\right\|_{L^1} \geq r_1,\quad	\lambda_2 m_2 f_2(R_1, r_2) \left\|h_2\right\|_{L^1} \geq r_2.
\end{equation}

Next, we show that the operator $\mathcal{T}$ satisfies the conditions of Theorem \ref{main_th} on $\overline{K}_{(r_1,r_2),(R_1,r_2^*)}$ and $\overline{K}_{(r_1,r_2^*),(R_1,R_2)}$. The restrictions of $\mathcal{T}$ to those sets are known to be compact and to map the corresponding sets into the cone $K$.

The same reasoning as in the proof of Theorem \ref{th_sol} shows that the first component of the restriction of $\mathcal{T}$ to $\overline{K}_{(r_1,r_2),(R_1,R_2)}$
 satisfies condition $(a)$ of Theorem \ref{main_th}. That is, for $x=(x_1,x_2)\in \overline{K}_{(r_1,r_2),(R_1,R_2)}$, we have
\[
\|\mathcal{T}_1 x\|_\infty \geq \|x_1\|_\infty, \, \text{if } \|x_1\|_\infty = r_1, 
\quad \text{and} \quad
\|\mathcal{T}_1 x\|_\infty \leq \|x_1\|_\infty, \, \text{if } \|x_1\|_\infty = R_1.
\]
Thus, condition $(a)$ is verified for $\mathcal{T}_1$ on both restrictions under consideration.

We now turn our attention to the second component. Take $x=(x_1,x_2)\in\overline{K}_{(r_1,r_2),(R_1,R_2)}$ with $\|x_2\|_\infty=R_2$. Then $c_1 r_1\leq x_1(t)\leq R_1$ and $c_2 R_2\leq x_2(t)\leq R_2$ for all $t\in[0,T]$. From \eqref{eq3.13} we obtain
\[
\mathcal{T}_2 x(t) 
\ge \lambda_2 \int_0^T G_2(t,s)\, g_2(x_1(s),x_2(s))\, ds
\ge \lambda_2 m_2 \int_0^T \frac{x_2(s)}{c_2 \lambda_2 m_2 T}\, ds =  \int_0^T \frac{x_2(s)}{c_2 T}\, ds
\geq R_2 = \|x_2\|_\infty,
\]
for all $ t \in [0,T] $. Thus $\|\mathcal{T}_2 x\|_\infty\geq\|x_2\|_\infty$ if $\|x_2\|_\infty=R_2$.

On the other hand, let $x=(x_1,x_2)\in \overline{K}_{(r_1,r_2),(R_1,R_2)}$ be such that $\|x_2\|_\infty = r_2^*$. Then $0 \leq x_1(t) \leq R_1$ and $c_2 r_2^* \leq x_2(t) \leq r_2^*$ for all $t \in [0,T]$. Since $\lambda_2 \in (0,\overline{\lambda}_2(\lambda_1))$, by $(H_1)$ we obtain
\begin{align*}
	\mathcal{T}_2 x(t) 
	&= \lambda_2 \int_0^T G_2(t,s)\bigl( h_2(s)\, f_2(x_1(s),x_2(s)) + g_2(s,x_1(s),x_2(s)) \bigr)\, ds \\
	&\leq \lambda_2 M_2  \left(\|h_2\|_{L^1} f_2(0,c_2 r_2^*) 
	+ T\max_{[0,T]\times[0,R_1(\lambda_1)]\times[c_2 r_2^*,\, r_2^*]} g_2(t,x_1,x_2)\right) \\
	&= \frac{\lambda_2}{\overline{\lambda}_2(\lambda_1)}\, r_2^* 
	< r_2^* = \|x_2\|_\infty,
\end{align*}
for all $t \in [0,T]$. Hence, $\|\mathcal{T}_2 x\|_\infty < \|x_2\|_\infty$ if  $\|x_2\|_\infty = r_2^*$.

Finally, note that if $x=(x_1,x_2)\in \overline{K}_{(r_1,r_2),(R_1,R_2)}$ is such that $\|x_2\|_\infty=r_2$, then $c_1 r_1 \leq x_1(t)\leq R_1$ and $c_2r_2\leq x_2(t)\leq r_2$ for all $t\in[0,T]$. Therefore, by using \eqref{eq3.15} we obtain
\[
\mathcal{T}_2 x(t) 
\ge \lambda_2 \int_0^T G_2(t,s)\, h_2(s)\, f_2(x_1(s),x_2(s))\, ds
\ge \lambda_2 m_2 f_2(R_1, r_2) \int_0^T h_2(s)\, ds
\geq r_2 = \|x_2\|_\infty,
\]
for all $ t \in [0,T] $, so that $\|\mathcal{T}_2 x\|_\infty \geq \|x_2\|_\infty$ if $\|x_2\|_\infty=r_2$. 

We have verified that conditions $(a)$ and $(b)$ hold for the second component of the restrictions $\mathcal{T}:\overline{K}_{(r_1,r_2),(R_1,r_2^*)} \to K$ and $\mathcal{T}:\overline{K}_{(r_1,r_2^*),(R_1,R_2)}\to K$, respectively. Therefore, Theorem \ref{main_th} ensures the existence of two fixed points $x^1=(x^1_1, x^1_2)$ and $x^2=(x^2_1, x^2_2)$ of $\mathcal{T}$ such that $r_1 \leq \|x^j_1\|_\infty \leq R_1, \, j=1,2,$ and $r_2 \leq \|x^1_2\|_\infty \leq r_2^* \leq \|x^2_2\|_\infty \leq R_2$. Observe that $\left\|x^1_2\right\|_{\infty}\neq r_2^*\neq \left\|x^2_2\right\|_{\infty}$, since $\left\|\mathcal{T}_2 x\right\|_{\infty}<\left\|x_2\right\|_{\infty}$ whenever $\left\|x_2\right\|_{\infty}=r_2^*$, which ensures that $x^1$ and $x^2$ are distinct fixed points. These correspond to the desired solutions of system \eqref{eq2.2}. \qed

Next, we introduce a new condition that will be essential for establishing the existence of four solutions and, when imposed in the context of Theorem~\ref{th_two_sol}, allows for an enhanced conclusion.

\begin{enumerate}
	\item[$(H_3)_i$] There exist $r^*>0$ and $L>0$ such that
	\begin{equation}\label{eq_bound_gi}
	\sup_{\begin{array}{c} t \in [0,T], x_i\in[c_i r^*,r^*], x_j\in[0,\infty) \\ (j\neq i)\end{array}}g_i(t,x_1,x_2)\leq L.
	\end{equation}
\end{enumerate}

\begin{corollary}
	\label{coro}
Let $(A_1)$, $(H_1)$, and $(H_2)$ hold. Assume that $g_1$ is sublinear and $g_2$ is superlinear at infinity satisfying $(H_3)_2$.

Then, there exists $\overline{\lambda}_2>0$ such that system \eqref{eq2.2} has, at least, two distinct positive periodic solutions for all $\lambda_1>0$ and  $0<\lambda_2<\overline{\lambda}_2$.
\end{corollary}

\noindent
{\bf Proof.} Since $g_2$ satisfies condition $(H_3)_2$, in the proof of Theorem~\ref{th_two_sol} we choose $r^*_2>0$ as the value $r^*$ given by $(H_3)_2$ and so there exists $L> 0$ such that the dependence on $\lambda_1$ in the choice of $\overline{\lambda}_2$ in \eqref{eq3.14} can be removed. It then suffices to define
\begin{equation*}
	\overline{\lambda}_2 := r^*_2 \frac{1}{M_2 \bigl(\|h_2\|_{L^1} f_2(0,c_2 r^*_2) + T L\bigr)}.
\end{equation*}
The remaining part of the proof follows analogously.
\qed

\begin{remark}
	\label{re3.1}
	Under the assumptions of Theorem~\ref{th_two_sol} or Corollary \ref{coro}, system \eqref{eq2.2} has at least one positive solution provided that $\lambda_1>0$ and $\lambda_2=\overline{\lambda}_2$. Indeed, following the proof of Theorem~\ref{th_two_sol}, the unique difference is that now $\|\mathcal{T}_2 x\|_\infty \leq \|x_2\|_\infty$ if  $\|x_2\|_\infty = r_2^*$, and thus it cannot be ensured that $x^1$ and $x^2$ are distinct solutions. 
\end{remark}

\begin{remark}
	Clearly, the roles of $g_1$ and $g_2$ in Theorem~\ref{th_two_sol} and Corollary \ref{coro} are interchangeable, yielding a symmetric conclusion.
\end{remark}

\begin{remark}
	Observe that condition $(H_3)_i$ obviously holds if $g_i$ is uniformly bounded with respect to the variables $t$ and $x_j$ ($j\neq i$), that is, if there exists a continuous function $L:\mathbb{R}^+\to\mathbb{R}^+$ such that 
	\[g_i(t,x_1,x_2)\leq L(x_i) \quad \text{ for all } t\in[0,T] \text{ and } x_1,x_2\geq 0. \]
	In particular, notice that if $g_2$ is non-increasing with respect to $x_1$, then it satisfies condition $(H_3)_2$. Analogously, when $g_1$ is non-increasing with respect to $x_2$, it fulfills condition $(H_3)_1$.
	
	Note that, in these cases, for each $r^*>0$ there exists $L(r^*)>0$ such that \eqref{eq_bound_gi} holds. However, it is enough that it be satisfied for a specific $r^*>0$. It is a routine to check that for the function $g_1$ given by
	\[g_1(x_1,x_2)=\dfrac{1}{(x_1-2)^2+e^{-x_2} } \]
	condition $(H_3)_1$ holds with $r^*=1$ and $L=1$, whereas $g_1(2,\cdot)$ is unbounded.
\end{remark}

\subsection{Superlinear nonlinearities}

Finally, it is also possible to establish the existence of four distinct solutions under superlinearity for certain configurations of $(\lambda_1,\lambda_2)$. To this end, it is necessary to impose condition $(H_3)_i$ on both components, as shown in our final main theorem below.

\begin{theorem}
	\label{th_four_sol}
	Let $(A_1)$, $(H_1)$ and $(H_2)$ hold. Assume that $g_i$ is superlinear at infinity fulfilling $(H_3)_i$ for each $i=1,2$.
	
	Then, there exist $\overline{\lambda}_i>0$, $i=1,2$ such that system \eqref{eq2.2} has, at least, four distinct positive periodic solutions for all $0<\lambda_i<\overline{\lambda}_i$, $i=1,2$.
\end{theorem}

\noindent
{\bf Proof.} Since $g_i$ satisfies condition $(H_3)_i$, there exist $r^*_i>0$ and $L^i(r^*_i)>0$ such that
\[\sup_{\begin{array}{c} t \in [0,T], x_i\in[c_i r^*,r^*], x_j\in[0,\infty) \\ (j\neq i)\end{array}}g_i(t,x_1,x_2)\leq L^i(r^*_i), \quad i=1,2.\]
We will find $\tilde{r}_i, \tilde{R}_i\in\mathbb{R}$ fulfilling $0<\tilde{r}_i<r^*_i<\tilde{R}_i$, $i=1,2$ such that Theorem \ref{main_th} applies to the restriction of $\mathcal{T}$ to $\overline{K}_{(\tilde{r}_1,\tilde{r}_2),(r^*_1,r^*_2)}$, $\overline{K}_{(\tilde{r}_1,r^*_2),(r^*_1,\tilde{R}_2)}$, $\overline{K}_{(r^*_1,r^*_2), (\tilde{R}_1,\tilde{R}_2)}$ and $\overline{K}_{(r^*_1,\tilde{r}_2),(\tilde{R}_1,r^*_2)}$. 

Set now
\begin{equation}
	\label{eq3.18}
	\overline{\lambda}_1:=r^*_1 \frac{1}{M_1 \bigl(\|h_1\|_{L^1} f_1(c_1 r^*_1, 0) +T L^1(r^*_1)\bigr)},\quad \overline{\lambda}_2:=r^*_2 \frac{1}{M_2 \bigl(\|h_2\|_{L^1} f_2(0,c_2 r^*_2) + T L^2(r^*_2)\bigr)}
\end{equation}
and take $\lambda_i\in(0,\bar{\lambda}_i)$, $i=1,2$. 
Since $g_i$ is superlinear at infinity, there exists $R_i\equiv R_i(\lambda_i)>r^*_i$ such that
\begin{equation}
	\label{eq3.17}
	g_i(t,x_1,x_2) \geq \frac{1}{c_i \lambda_i m_i T} x_i, 
	\quad \text{for all } x_i \ge c_i R_i,\, x_j \ge 0, \text{ and } t \in [0,T], \quad i=1,2.
\end{equation}
Choose $r_1\equiv r_1(\lambda_1,\lambda_2)\in (0,r^*_1)$ and $r_2\equiv r_2(\lambda_1, \lambda_2)\in(0,r^*_2)$ fulfilling
\begin{equation}
	\label{eq3.19}
	\lambda_1 m_1 f_1(r_1,R_2) \|h_1\|_{L^1}\geq r_1, \quad 	\lambda_2 m_2 f_2(R_1,r_2)\|h_2\|_{L^1}\geq r_2.
\end{equation}

We check the conditions in Theorem \ref{main_th} for the first component of the operator. Let $x=(x_1,x_2)\in\overline{K}_{(r_1,r_1),(R_1,R_2)}$ be such that $\|x_1\|_\infty=R_1$. Then $c_1 R_1\leq x_1(t)\leq  R_1$ and $c_2 r_2\leq x_2(t)\leq  R_2$ for all $t\in[0,T]$. By using \eqref{eq3.17} we obtain
\begin{align*}
	\mathcal{T}_1 x(t) 
	\geq\lambda_1 \int_0^T G_1(t,s)g_1(s,x_1(s),x_2(s))\, ds \geq \lambda_1 m_1 \int_{0}^{T} \frac{x_1(s)}{c_1\lambda_1 m_1 T} dt = \int_{0}^{T} \frac{x_1(s)}{c_1 T} dt \geq R_1=\|x_1\|_\infty, 
\end{align*}
for all $t\in[0,T]$. Hence $\|\mathcal{T}_1 x\|_\infty \geq\|x_1\|_\infty$ if $\|x_1\|_\infty=R_1$.

On the other hand, if $x = (x_1,x_2) \in \overline{K}_{(r_1,r_2),(R_1,R_2)}$ is such that $\|x_1\|_\infty = r_1^*$, then $c_1 r_1^* \leq x_1(t) \leq r_1^*$ and $0 \leq x_2(t) \leq R_2$ for all $t \in [0,T]$. Since $0<\lambda_1<\overline{\lambda}_1$ the following inequalities hold
\begin{align*}
	\mathcal{T}_1 x(t) 
	&= \lambda_1 \int_0^T G_1(t,s)\bigl( h_1(s)\, f_1(x_1(s),x_2(s)) + g_1(s,x_1(s),x_2(s)) \bigr)\, ds \\
	&\leq \lambda_1 M_1 \left(\|h_1\|_{L^1} f_1(c_1 r_1^*,0) 
	+T \max_{[0,T]\times[c_1 r_1^*,\, r_1^*]\times[0,R_2(\lambda_2)]} g_1(t,x_1,x_2)\right) \\
	&\leq \lambda_1 M_1 \left(\|h_1\|_{L^1} f_1(c_1 r_1^*,0) 
	+T L^1(r^*_1)\right) = \frac{\lambda_1}{\overline{\lambda}_1}\, r_1^* 
	< r_1^* = \|x_1\|_\infty,
\end{align*}
for all $t \in [0,T]$. Therefore, $\|\mathcal{T}_1 x\|_\infty < \|x_1\|_\infty$ if $\|x_1\|_\infty = r_1^*$. 

Finally, let $x = (x_1,x_2) \in \overline{K}_{(r_1,r_2),(R_1,R_2)}$ be such that $\|x_1\|_\infty = r_1$. Then $c_1 r_1\leq x_1(t) \leq r_1$ and $c_2 r_2\leq x_2(t) \leq R_2$ for all $t \in [0,T]$. Thus, by \eqref{eq3.19} and $(H_1)$ we obtain
\[
\mathcal{T}_1 x(t) 
\ge \lambda_1 \int_0^T G_1(t,s)\, h_1(s)\, f_1(x_1(s),x_2(s))\, ds
\ge \lambda_1 m_1 f_1(r_1,R_2) \int_0^T h_1(s)\, ds
\geq r_1 = \|x_1\|_\infty,
\]
for all $ t \in [0,T] $, so that $\|\mathcal{T}_1 x\|_\infty \geq \|x_1\|_\infty$ if $\|x_1\|_\infty = r_1$. 

Symmetric arguments yield the same behavior for the second component. Namely,
\[
\|\mathcal{T}_2 x\|_\infty \geq \|x_2\|_\infty \quad \text{if } \|x_2\|_\infty = r_2 \text{ or } \|x_2\|_\infty = R_2, \text{ and } \|\mathcal{T}_2 x\|_\infty < \|x_2\|_\infty \quad \text{if } \|x_2\|_\infty = r_2^*.
\]

In this way, each restriction of $\mathcal{T}$ is under the assumptions of Theorem~\ref{main_th}$:$ on $\overline{K}_{(r_1,r_2),(r_1^*,r_2^*)}$, condition $(a)$ holds for both components; on $\overline{K}_{(r_1^*,r_2),(R_1, r_2^*)}$, condition $(a)$ holds for the second component and $(b)$ for the first; on $\overline{K}_{(r_1,r_2^*),(r_1^*,R_2)}$, condition $(a)$ holds for the first component and $(b)$ for the second; finally, on $\overline{K}_{(r_1^*,r_2^*),(R_1, R_2)}$, condition $(b)$ holds for both components. Hence, there exists four distinct fixed points $x^1=(x_1^1,x_2^1)$, $x^2=(x_1^2,x_2^2)$, $x^3=(x_1^3,x_2^3)$ and $x^4=(x_1^4,x_2^4)$ of $\mathcal{T}$ such that 
\[r_1 \leq \|x_1^j\|<r^*_1<\|x_1^k\| \leq R_1 \quad \text{ and }\quad r_2 \leq \|x_2^l\|<r^*_2<\|x_{2}^s\|\leq  R_2 , \]
with $ j\in\{1,2\}$, $k\in\{3,4\}$, $l\in\{1,4\}$ and $s\in\{2,3\}$. This completes the proof, as the desired solutions have been obtained as fixed points of $\mathcal{T}$. \qed

\begin{remark}
Arguing as in Remark~\ref{re3.1}, and under the assumptions of Theorem~\ref{th_four_sol}, if $\lambda_1 \in (0,\overline{\lambda}_1)$ and $\lambda_2 = \overline{\lambda}_2$, system \eqref{eq2.2} has at least two distinct solutions. The same conclusion holds when $\lambda_1 = \overline{\lambda}_1$ and $\lambda_2 \in (0,\overline{\lambda}_2)$. Furthermore, if $\lambda_1 = \overline{\lambda}_1$ and $\lambda_2 = \overline{\lambda}_2$, one can ensure the existence of at least one solution.
\end{remark}


\section{Examples}

The example below illustrates the application of the main results.

\begin{example}
We establish sufficient conditions for the existence of positive $T$-periodic solutions to systems of the form
\begin{equation}
	\label{eq3.20}
	\left\{
	\begin{aligned}
		& x_1'' + a_1(t)x_1 = \lambda_1\sqrt{(x_1^2+x_2^2)^{-\alpha}}  + \lambda_1 \bigl(\psi_1(x_2) + x_1^{\beta_1}\bigr), \quad t \in [0,T], \\
		& x_2'' + a_2(t)x_2 = \lambda_2 \sqrt{(x_1^2+x_2^2)^{-\alpha}}  + \lambda_2 \bigl(\psi_2(x_1) + x_2^{\beta_2}\bigr), \quad t \in [0,T],
	\end{aligned}
	\right.
\end{equation}
under assumption $(A_1)$, where $\alpha>0$ and $\psi_i:\mathbb{R}^+\to \mathbb{R}^+$, $i=1,2$, are continuous and bounded functions. The number of such solutions will depend on the parameters $\beta_1$ and $\beta_2$.

Observe that system \eqref{eq3.20} fits into the form of \eqref{eq2.2} with
\[
h_i(t)=1, \quad f_i(x_1,x_2)=\sqrt{(x_1^2+x_2^2)^{-\alpha}} , \quad \text{and} \quad 
g_i(t,x_1,x_2)=\psi_i(x_j)+x_i^{\beta_i}, \quad (j\neq i), \quad i=1,2.
\]

It is straightforward to verify that conditions $(H_1)$ and $(H_2)$ hold. Furthermore, since $\psi_i$ is bounded, then
\[
(g_i)_\infty=
	\left\{
\begin{aligned}
	&0, \quad && \text{if } 0<\beta_i<1,\\
&\infty, \quad && \text{if } \beta_i>1,
\end{aligned}
\right.
\qquad i=1,2.
\]
Moreover, each $g_i$ satisfies condition $(H_3)_i$.

As a consequence of Theorems \ref{th_sol} and \ref{th_four_sol}, together with Corollary \ref{coro}, the following statements hold:

\begin{enumerate}
	\item[(a)] If $\beta_i\in(0,1)$ for $i=1,2$, then system \eqref{eq3.20} admits at least one positive $T$-periodic solution for all $\lambda_i>0$, $i=1,2$.
	
	\item[(b)] If $\beta_1\in(0,1)$ and $\beta_2>1$, then there exists $\overline{\lambda}_2>0$ such that system \eqref{eq3.20} admits at least two distinct positive $T$-periodic solutions  for all $\lambda_1>0$ and $0<\lambda_2<\overline{\lambda}_2$.
	
	\item[(c)] If $\beta_i>1$ for $i=1,2$, then there exist $\overline{\lambda}_i>0$, $i=1,2$, such that system \eqref{eq3.20} admits at least four distinct positive $T$-periodic solutions  for all $0<\lambda_i<\overline{\lambda}_i$, $i=1,2$.
\end{enumerate}

Let $L_i>0$ denote the upper bound of $\psi_i$. Then, $0\leq \psi_i(x)\leq L_i$ for all $x\in\mathbb{R}^+$, $i=1,2$. By fixing $r_1^*=r_2^*=1$, the constants in $(b)$ and $(c)$ can be chosen as
\[
\overline{\lambda}_i=\bigl(M_i T \left(c_i^{-\alpha}+L_i+1\right)\bigr)^{-1}, \quad i=1,2.
\]

In this case, the solutions $x^1=(x_1^1,x_2^1)$ and $x^2=(x_1^2,x_2^2)$ obtained in situation $(b)$ satisfy $0 < \|x_1^j\|_\infty, \, j=1,2$ and $0 < \|x_2^1\|_\infty < 1 < \|x_2^2\|_\infty$.
Similarly, the solutions $x^1, x^2, x^3, x^4$ obtained in situation $(c)$ satisfy $0<\|x_1^j\|_\infty<1<\|x_1^k\|_\infty,$ $0<\|x_2^l\|_\infty<1<\|x_2^s\|_\infty,$ with $j\in\{1,2\}$, $k\in\{3,4\}$, $l\in\{1,4\}$, and $s\in\{2,3\}$.
%
\end{example}

%

The requirements of conditions $(H_3)_1$ and $(H_3)_2$ in Theorem \ref{th_four_sol} may be weakened.
Indeed, following its proof, in order to still guarantee the existence of four distinct solutions, it suffices to choose 
	$(\lambda_1,\lambda_2)\in(0,\overline{\lambda}_1)\times(0,\overline{\lambda}_2)$, where 
	$(\overline{\lambda}_1, \overline{\lambda}_2)\in \mathbb{R}^2_+:=(0,\infty)^2$ is a solution of the system of nonlinear inequalities
	\begin{equation}
		\label{eq3.21}
		\begin{aligned}
			\overline{\lambda}_1 &\leq r^*_1\Biggl[M_1 \Bigl(\|h_1\|_{L^1} f_1(c_1r^*_1,0) + 
			T \max_{[0,T]\times[c_1r^*_1, r^*_1]\times[0,R_2(\overline{\lambda}_2)]} g_1(t,x_1,x_2)\Bigr)\Biggr]^{-1},\\
			\overline{\lambda}_2 &\leq r^*_2\Biggl[M_2 \Bigl(\|h_2\|_{L^1} f_2(0,c_2r^*_2) + 
			T \max_{[0,T]\times[0,R_1(\overline{\lambda}_1)]\times[c_2r^*_2, r^*_2]} g_2(t,x_1,x_2)\Bigr)\Biggr]^{-1},
		\end{aligned}
	\end{equation}
	for fixed $r^*_1,r^*_2>0$.
There exist particular examples of functions $g_i$ for which \eqref{eq3.21} admits a solution 
$(\overline{\lambda}_1, \overline{\lambda}_2)\in\mathbb{R}^2_+$, without fulfilling $(H_3)_1$ and $(H_3)_2$.  A concrete example is given below.

\begin{example}
	A pair $(\overline{\lambda}_1,\overline{\lambda}_2)\in\mathbb{R}^2_+$ can be chosen such that, for all $\lambda_i\in(0,\overline{\lambda}_i)$, $i=1,2$, there exists four distinct $T$-periodic solutions with strictly positive components for the system
	\begin{equation}
		\label{eq3.22}
		\left\{
		\begin{aligned}
			& x_1'' + a_1(t)x_1 = \lambda_1\sqrt{(x_1^2+x_2^2)^{-\alpha_1}}  + \lambda_1 \bigl( x_1^{\beta^1_1}+x_2^{\beta^1_2} \bigr), \quad t \in [0,T], \\
			& x_2'' + a_2(t)x_2 = \lambda_2 \sqrt{(x_1^2+x_2^2)^{-\alpha_2}}  + \lambda_2 \bigl(x_1^{\beta^2_1}+x_2^{\beta^2_2} \bigr), \quad t \in [0,T],
		\end{aligned}
		\right.
	\end{equation}
 where $\alpha_1,\alpha_2>0$, $\beta^1_1, \, \beta^2_2>1$, $\beta^1_2<\beta^2_2-1$, $\beta^2_1<\beta^1_1-1$ and under assumption $(A_1)$.

Observe that system \eqref{eq3.22} may also be written in the form of \eqref{eq2.2}, with
\[
h_i(t)=1, \quad 
f_i(x_1,x_2)=\sqrt{(x_1^2+x_2^2)^{-\alpha_i}}, \quad \text{and} \quad 
g_i(t,x_1,x_2)=x_i^{\beta^i_i}+x_j^{\beta^i_j}, \quad (j \neq i), \quad i=1,2.
\]
Since $\beta^1_1, \, \beta^2_2>1$, it follows that 
\[
\lim_{x_1\to \infty}\,\min_{x_2\in[0,\infty)}\frac{g_1(t,x_1,x_2)}{x_1}
=\lim_{x_1\to \infty} x_1^{\beta_1^1-1}=+\infty \,\, \text{ and } \,\, \lim_{x_2\to \infty}\,\min_{x_1\in[0,\infty)}\frac{g_2(t,x_1,x_2)}{x_2}
=\lim_{x_2\to \infty} x_2^{\beta_2^2-1}=+\infty.
\]
Hence, each $g_i$ is superlinear for $i=1,2$. Moreover, for $\lambda_1,\lambda_2>0$, the quantities $R_1(\lambda_1)$ and $R_2(\lambda_2)$ in the proof of Theorem \ref{th_four_sol} are determined by
\[
c_1^{\beta^1_1}R_1(\lambda_1)^{\beta^1_1}=\frac{1}{\lambda_1 m_1 T}R_1(\lambda_1),
\quad 
c_2^{\beta^2_2}R_2(\lambda_2)^{\beta^2_2}=\frac{1}{\lambda_2 m_2 T}R_2(\lambda_2),
\]
since these correspond to the minimal values satisfying \eqref{eq3.17}. Consequently,
\begin{equation}
	\label{eq3.23}
R_1(\lambda_1)=\left(\frac{1}{\lambda_1 m_1 T\,c_1^{\beta^1_1}}\right)^{\!\frac{1}{\beta^1_1-1}}, \quad R_2(\lambda_2)=\left(\frac{1}{\lambda_2 m_2 T\,c_2^{\beta^2_2}}\right)^{\!\frac{1}{\beta^2_2-1}}.
\end{equation}

It is immediate that $R_1(\lambda_1)$ and $R_2(\lambda_2)$ are decreasing with respect to $\lambda_1$ and $\lambda_2$, respectively, and
\[
\lim_{\lambda_i\to 0^+} R_i(\lambda_i)=+\infty, \quad i=1,2.
\]
Therefore, for any fixed $r_i^*>0$, one can choose $\lambda_i>0$ sufficiently small so that $R_i(\lambda_i)>r_i^*$, $i=1,2$.

Using \eqref{eq3.23}, we obtain
\[
\max_{[0,T]\times[c_1r^*_1,r^*_1]\times[0,R_2(\lambda_2)]} g_1(t,x_1,x_2)=(r^*_1)^{\beta^1_1}+\left(\frac{1}{\lambda_2 m_2 T\,c_2^{\beta^2_2}}\right)^{\!\frac{\beta^1_2}{\beta^2_2-1}}=\frac{(r^*_1)^{\beta^1_1}\bigl(\lambda_2 m_2 T c_2^{\beta^2_2}\bigr)^{\frac{\beta^1_2}{\beta^2_2-1}}+1}{\bigl(\lambda_2 m_2 T c_2^{\beta^2_2}\bigr)^{\frac{\beta^1_2}{\beta^2_2-1}}},
\]
\[
\max_{[0,T]\times[0,R_1(\lambda_1)]\times[c_2r^*_2, r^*_2]} g_2(t,x_1,x_2)=(r^*_2)^{\beta^2_2}+\left(\frac{1}{\lambda_1 m_1 T\,c_1^{\beta^1_1}}\right)^{\!\frac{\beta^2_1}{\beta^1_1-1}}=\frac{(r^*_2)^{\beta^2_2}\bigl(\lambda_1 m_1 T c_1^{\beta^1_1}\bigr)^{\frac{\beta^2_1}{\beta^1_1-1}}+1}{\bigl(\lambda_1 m_1 T c_1^{\beta^1_1}\bigr)^{\frac{\beta^2_1}{\beta^1_1-1}}}.
\]

Thus, the nonlinear system of inequalities \eqref{eq3.21} reduces in this setting to
\begin{equation}
	\label{eq3.24}
\overline{\lambda}_1\leq \frac{r^*_1}{M_1 T} \frac{\bigl(\overline{\lambda}_2 m_2 T c_2^{\beta_2^2}\bigr)^{\frac{\beta^1_2}{\beta^2_2-1}}}{\bigl[(c_1 r^*_1)^{-\alpha}+(r^*_1)^{\beta^1_1}\bigr]\bigl(\overline{\lambda}_2 m_2 T c_2^{\beta_2^2}\bigr)^{\frac{\beta^1_2}{\beta^2_2-1}}+1}, \,\, \overline{\lambda}_2\leq \frac{r^*_2}{M_2 T} \frac{\bigl(\overline{\lambda}_1 m_1 T c_1^{\beta_1^1}\bigr)^{\frac{\beta^2_1}{\beta^1_1-1}}}{\bigl[(c_2 r^*_2)^{-\alpha}+(r^*_2)^{\beta^2_2}\bigr]\bigl(\overline{\lambda}_1 m_1 T c_1^{\beta_1^1}\bigr)^{\frac{\beta^2_1}{\beta^1_1-1}}+1}.
\end{equation}

A pair $(\overline{\lambda}_1,\overline{\lambda}_2)\in\mathbb{R}^2_+$ satisfying \eqref{eq3.24} exists due to the assumptions $\beta^1_2<\beta^2_2-1$ and $\beta^2_1<\beta^1_1-1$. Indeed, \eqref{eq3.24} can be written in the form
\begin{equation}
	\label{eq3.25}
	x\leq\frac{a_1y^{\delta_1}}{b_1y^{\delta_1}+1},
	\quad
	y\leq\frac{a_2x^{\delta_2}}{b_2x^{\delta_2}+1},
\end{equation}
with $a_i,b_i>0$ and $\delta_i\in(0,1)$, $i=1,2$.

System \eqref{eq3.25} admits infinitely many solutions $(\bar{x}, \bar{y})\in\mathbb{R}^2_+$. Indeed
\[\lim_{y\to 0^+}\big(b_1y^{\delta_1}+1\big)=1,
\quad
\lim_{x\to 0^+}\big(b_2x^{\delta_2}+1\big)=1,\]
and hence there exists $(\varepsilon_1,\varepsilon_2)\in \mathbb{R}^2_+$ such that
\[b_1y^{\delta_1}+1<\tfrac{3}{2}, \, \text{ for } y\in(0,\varepsilon_1) \quad \text{ and } \quad b_2x^{\delta_2}+1<\tfrac{3}{2}, \, \text{ for } x\in(0,\varepsilon_2).\]
Therefore, any $(x,y)\in(0,\varepsilon_1)\times(0,\varepsilon_2)$ satisfying
\begin{equation}
	\label{eq3.26}
	x\leq \tfrac{2}{3} a_1 y^{\delta_1}, 
	\quad 
	y\leq \tfrac{2}{3} a_2 x^{\delta_2}
\end{equation}
is also a solution of \eqref{eq3.25}.

Substituting the second inequality into the first, we obtain
\[
x \leq \tfrac{2}{3}a_1\left(\tfrac{2}{3}a_2x^{\delta_2}\right)^{\delta_1}
= \tfrac{2}{3}\left(\tfrac{2}{3}\right)^{\delta_1} a_1 a_2^{\delta_1} x^{\delta_1\delta_2},
\]
which is equivalent to
\begin{equation}
	\label{eq3.27}
	x^{1-\delta_1\delta_2} \leq \tfrac{2}{3}\left(\tfrac{2}{3}\right)^{\delta_1} a_1 a_2^{\delta_1}.
\end{equation}
Since $1-\delta_1\delta_2>0$, the left-hand side tends to zero as $x\to 0^+$, and thus one can choose $x_0>0$ satisfying \eqref{eq3.27}. Setting, for instance, $y_0:=\tfrac{1}{3}a_2 x_0^{\delta_2}$, it follows that $(x_0,y_0)$ is a solution of \eqref{eq3.26}, and hence of \eqref{eq3.25}.

Therefore, letting $(\overline{\lambda}_1,\overline{\lambda}_2)\in\mathbb{R}^2_+$ be a solution of \eqref{eq3.24} and choosing $\lambda_i\in(0,\overline{\lambda}_i)$, $i=1,2$, sufficiently small so that $R_i(\lambda_i)>r_i^*$, with fixed $r_i^*>0$, it follows from the proof of Theorem \ref{th_four_sol} the existence of four $T$-periodic solutions with positive components $x^1, x^2, x^3, x^4$ satisfying $0<\|x_1^j\|_\infty<r^*_1<\|x_1^k\|_\infty,$ $0<\|x_2^l\|_\infty<r^*_2<\|x_2^s\|_\infty,$ with $j\in\{1,2\}$, $k\in\{3,4\}$, $l\in\{1,4\}$, and $s\in\{2,3\}$.
\end{example}

\section*{Acknowledgements}
L. M. Fern\'andez-Pardo acknowledges the financial support of the Spanish Ministry of Science, Innovation and Universities (Grant reference FPU24/01545).

J. Rodr\'iguez--L\'opez has been partially supported by Ministerio de Ciencia y Tecnología (Spain), AEI and Feder, grant PID2020-113275GB-I00, and Xunta de Galicia, Spain, Project ED431C 2023/12.

\end{document}